\documentclass[12pt]{amsart}
\usepackage{amssymb}
\usepackage[margin=1in]{geometry}

\title{Weak* tensor products for von Neumann algebras}
\author{Matthew Wiersma}
\address{Department of Pure Mathematics, University of Waterloo, Waterloo, ON, Canada N2L 3G1}
\email{mwiersma@uwaterloo.ca}

\newtheorem{theorem}{Theorem}[section]
\newtheorem{corollary}[theorem]{Corollary}
\newtheorem{prop}[theorem]{Proposition}
\newtheorem{lemma}[theorem]{Lemma}
\newtheorem{proposition}[theorem]{Proposition}

\theoremstyle{definition}
\newtheorem{defn}[theorem]{Definition}
\newtheorem{example}[theorem]{Example}
\newtheorem{remark}[theorem]{Remark}

\newcommand{\fn}{\!:}
\newcommand{\C}{\mathbb C}
\newcommand{\R}{{\mathbb R}}
\newcommand{\N}{{\mathbb N}}

\newcommand{\lla}{\left\langle}
\newcommand{\rra}{\right\rangle}
\newcommand{\mc}{\mathcal}
\newcommand{\mf}{\mathfrak}
\newcommand{\tn}{\textnormal}

\newcommand{\wm}{\overline\otimes_{w^*\tn{-max}}}
\newcommand{\oo}{\overline{\otimes}}

\begin{document}

\begin{abstract}
The category of $C^*$-algebras is blessed with many different tensor products. In contrast, virtually the only tensor product ever used in the category of von Neumann algebras is the normal spatial tensor product. We propose a definition of what a generic tensor product in this category should be. We call these {\it weak* tensor products}.

For von Neumann algebras $M$ and $N$, there are, in general, many choices of weak* tensor completions of the algebraic tensor product $M\odot N$. In fact, we demonstrate that $M$ has the property that $M\odot N$ has a unique weak* tensor product completion for every von Neumann algebra $N$ if and only if $M$ is completely atomic, i.e., is a direct product of type I factors. This in particular implies that even abelian von Neumann algebras need not have this property. As an application of the theory developed throughout the paper, we construct $2^{\mf c}$ nonequivalent weak* tensor product completions of $L^\infty(\R)\odot L^\infty(\R)$.


\end{abstract}

\maketitle

\section{Introduction}

In the category of $C^*$-algebras, a tensor product $A\otimes_{\alpha} B$ of two $C^*$-algebras $A$ and $B$ is any completion of the algebraic tensor product $A\odot B$ with respect to a $C^*$-norm $\|\cdot\|_\alpha$. The two most natural choices of tensor products of $C^*$-algebras are the maximal tensor product $A\otimes_{\max} B$ and minimal tensor product $A\otimes_{\min} B$. True to their names, these are the ``largest'' and ``smallest'' tensor product of $A$ and $B$ in the sense that the identity map on $A\odot B$ extends to $C^*$-quotients
$$ A\otimes_{\max} B\to A\otimes_\alpha B\to A\otimes_{\min} B$$
for any other tensor product $A\otimes_\alpha B$ of $A$ and $B$.

Though the maximal and minimal tensor products are the most commonly studied tensor products in the category of $C^*$-algebras, they are by no means the only tensor products studied. For example, the binormal $C^*$-tensor product $M\otimes_{\mathrm{bin}} N$ of von Neumann algebras $M$ and $N$ is studied by Effros and Lance in \cite{el}. More recently, Ozawa and Pisier constructed a continuum of $C^*$-norms on $B(H)\odot B(H)$ (see \cite{op}), where $H$ is the infinite dimensional separable Hilbert space, and the author constructed a continuum of $C^*$-norms on algebraic tensor products of certain group $C^*$-algebras (see \cite{w}).

In contrast to $C^*$-algebras, virtually the only tensor product in the category of von Neumann algebras ever studied is the normal spatial tensor product.
Further, there is not a concept of what other tensor products in this category should be. We propose that in the category of von Neumann algebras, a generic tensor product of von Neumann algebras $M$ and $N$ should be a von Neumann algebra $S$ which contains a weak* dense copy of the algebraic tensor product $M\odot N$ such that $M$ and $N$ are identifiable as von Neumann algebras with the copies of $M\otimes 1$ and $1\otimes N$ in $S$, respectively. We call such a von Neumann algebra $S$ a {\it weak* tensor product} of $M$ and $N$. In section 2 of this paper, we define this concept more rigorously and investigate weak* tensor products of factors.

Similar to the case of $C^*$-algebras, there is a ``largest'' weak* tensor product completion of $M\odot N$  for two von Neumann algebras $M$ and $N$. We briefly draw some connections between this ``largest'' weak*-tensor product and related structures in Section 3. Surprisingly, in general there is no ``smallest'' weak* tensor product despite the normal spatial tensor product being defined analogously to the minimal tensor product of $C^*$-algebras.

Tensor products play an invaluable role within the field of $C^*$-algebras and many properties of $C^*$-algebras (such as nuclearity (see \cite[Definition 11.4]{p}), exactness (see \cite[Chapter 17]{p}), and the WEP (see \cite[Proposition 15.3]{p})) are either defined or have a characterization in terms of tensor products. Perhaps weak* tensor products could play a similar role within von Neumann algebras in the future. Recall that a $C^*$-algebra $A$ is nuclear if and only if the algebraic tensor product $A\odot B$ has a unique $C^*$-completion for every second $C^*$-algebra $B$. In section 4, we completely characterize the analogous property for weak* tensor products. The class of nuclear $C^*$-algebras is large and contains many interesting examples. In contrast, a von Neumann algebra $M$ has the property that $M\odot N$ has a unique weak* tensor product completion for every von Neumann algebra $N$ if and only if $M$ is the direct product of type I factors. In particular, this implies that even abelian von Neumann algebras need not admit this property.

In the final section, we apply the theory developed throughout to studying weak* tensor product completions of $L^\infty(\R)\odot L^\infty(\R)$. In particular, we construct $2^{\mf c}$ nonequivalent weak* tensor product completions of $L^\infty(\R)\odot L^\infty(\R)$ and note that a generic weak* tensor product completion of $L^\infty(\R)\odot L^\infty(\R)$ need not have a separable predual, despite $L^\infty(\R)$ having a separable predual.

\section{Definitions and study of factors}\label{definitions section}

In this section we make precise the notion of weak* tensor products and examine particular examples involving factors. In particular, we show that if $M$ is a factor of type II or III, then $M\odot M'$ does not admit a unique weak* tensor product completion. The case when $M=B(H)$ for some Hilbert space $H$ is completely different, and we show that $B(H)$ has the property that $B(H)\odot N$ has a unique weak* tensor product completion for each von Neumann algebra $N$.


\begin{defn}
Let $M$ and $N$ be von Neumann algebras and suppose that $\alpha\fn M\odot N\to B(H)$ is an injective $*$-representation on some Hilbert space $H$ such that $\alpha|_{M}$ and $\alpha|_{N}$ are weak* homoemorphisms onto their respective ranges. (Here we identify $M$ and $N$ with $M\otimes 1$ and $1\otimes N$, respectively). The {\it weak* tensor product} $M\overline\otimes_\alpha N$ is defined to be the weak* closure of $\alpha(M\odot N)$ in $B(H)$. Such a weak* tensor product $M\oo_\alpha N$ is also called a {\it weak* tensor product completion} of $M\odot N$.
\end{defn}

Of course the normal spatial tensor product of von Neumann algebras is a particular example of a weak* tensor product. Indeed, if $M\subset B(H)$ and $N\subset B(K)$ are von Neumann algebras, then the canonical inclusion $\iota$ of $M\odot N$ into $B(H\otimes K)$ satisfies the conditions imposed by the above  definition and $M\overline{\otimes}N$ is the weak* closure of $\iota(M\odot N)$ in $B(H\otimes K)$.

The following example gives another construction of a weak* tensor product for factors.

\begin{example}\label{factor example}
Let $M\subset B(H)$ be a factor. It is an early result of Murray and von Neumann that the multiplication map $m\fn a\otimes b\mapsto ab$ from $M\odot M'\to B(H)$ is injective (see \cite{mv}). Observe that $a\in B(H)$ commutes with $M\cdot M'$ if and only if $a\in \C1$ as both $M$ and $M'$ are subsets of $M\cdot M'$. Hence, $M\overline\otimes_m M'=B(H)$.
\end{example}

Recall that if $M$ is factor, then $M\overline{\otimes}M'$ is a factor of the same type. Hence, if $M$ is a factor of type II or III and $m$ is as in Example \ref{factor example}, then $M\overline\otimes_m M'$ and $M\overline{\otimes}M'$ are not $*$-isomorphic. 
On the other hand, if $M=B(H)$ is a type I factor then $M\oo_m M'$ is trivially canonically $*$-isomorphic to $M\oo M'$ since $M'=\C1$. This motivates us to define a means of comparison for weak* tensor products.

\begin{defn}
Let $M$ and $N$ be von Neumann algebras. Two weak* tensor products $M\overline\otimes_\alpha N$ and $M\overline\otimes_\beta N$ of $M$ and $N$ are {\it equivalent} if the map $\alpha(a\otimes b)\mapsto \beta(a\otimes b)$ for $a\in M$, $b\in N$ extends to a (normal) $*$-isomorphism of $M\overline\otimes_\alpha N$ onto $M\overline\otimes_\beta N$. 
\end{defn}

%

Let $A$ be a $*$-algebra and $\pi\fn A\to B(H)$, $\sigma\fn A\to B(K)$ be two $*$-representations. We recall that the map $\sigma(a)\mapsto \pi(a)$ extends to a normal $*$-homomorphism from $\pi(A)''$ to $\sigma(A)''$ if and only if $\pi$ is quasi-contained in $\sigma$, i.e., if and only if $\pi$ is unitarily equivalent to a subrepresentation of some amplification of $\sigma$. In particular, this immediately gives that two weak* tensor products $M\overline\otimes_\alpha N$ and $M\overline\otimes_\beta N$ are equivalent if and only if $\alpha$ and $\beta$ are quasi-equivalent (i.e., each is quasi-contained in the other) and, further, that the identity map on $M\odot N$ extends to a normal $*$-homomorphism from $M\overline\otimes_{\alpha} N$ to $M\overline\otimes_\beta N$ if and only if $\beta$ is quasi-equivalent to a subrepresentation of $\alpha$. This observation allows us to see that there is a universal weak* tensor product.

\begin{prop}\label{max}
Let $M$ and $N$ be von Neumann algebras. There exists a unique (up to equivalence) weak* tensor product $M\overline\otimes_\alpha N$ of $M$ and $N$ with the property that if $M\overline\otimes_\beta N$ is any other weak* tensor product, then the identity map on $M\odot N$ extends to a normal $*$-homomorphism from $M\overline\otimes_\alpha N$ onto $M\overline\otimes_\beta N$.
\end{prop}

\begin{proof}
The uniqueness of such a weak* tensor product is clear, so we focus on showing existence. Since the collection of all quasi-equivalence classes of representations of $M\odot N$ forms a set, the collection of equivalence classes of weak* tensor products of $M$ and $N$ must form a set. Let $\{M\overline\otimes_{\alpha_i}N : i\in I\}$ be the set formed by choosing one representative from each equivalence class, and define $\alpha=\bigoplus_{i\in I} \alpha_i$. Then, since $\alpha_i$ is a subrepresentation of $\alpha$ for every $i\in I$, it is clear that $M\overline\otimes_\alpha N$ has the desired universal property among weak* tensor products of $M$ and $N$.
\end{proof}

\begin{defn}
Let $M$ and $N$ be von Neumann algebras. The {\it maximal} (or {\it universal}) {\it weak* tensor product} of $M$ and $N$ is the weak* tensor product $M\overline\otimes_\alpha N$ of $M$ and $N$ with universal property described in Proposition \ref{max}. This weak* tensor product is denoted $M\wm N$.
\end{defn}


It is readily verified that this weak* tensor product is both commutative and associative. It is also easily checked that this weak* tensor product has the projectivity property since if $K$ is any weak* closed ideal of a von Neumann algebra $M$, then $M=N\oplus K$ where $N=M/K$.

We find it interesting to observe that if $M\subset B(H)$ is a factor of type II or III and $m\fn M\odot M'\to B(H)$ is the multiplication map, then $M\wm M'$ is not a factor since both $B(H)=M\oo_m M'$ and $M\oo M'$ are normal quotients of $M\wm M'$.

We find it useful to think of the normal spatial and maximal weak* tensor products as being analogues of the minimal and maximal tensor products of $C^*$-algebras. Recall that the minimal tensor product $A\otimes_{\min} B$ of $C^*$-algebras $A$ and $B$ has the property that if $A\otimes_{\alpha} B$ is any other $C^*$-completion of $A\odot B$, then the identity map on $A\odot B$ extends to a $*$-homomorphism from $A\otimes_{\alpha} B\to A\otimes_{\min} B$.

\begin{remark}
The examples we have already studied show that the analogue of this property does not hold for the normal spatial tensor product among the class of weak* tensor products.
Indeed, let $M\subset B(H)$ be a factor of type II or III and $m\fn M\odot M'\to B(H)$ the multiplication operator. Then $M\oo_{m} M'=B(H)$ is not a normal quotient of $M\oo M'$ since $B(H)$ is a type I factor but $M\oo M'$ is a factor not of type I. However, it is interesting to note that both $M\oo M'$ and $M\oo_m M$, being factors, are each minimal among the weak* tensor products of $M$ and $M'$.
\end{remark}

We finish this section with a brief discussion of an analogue of nuclearity for $C^*$-algebras phrased in terms of weak* tensor products. Typically the injectivity of von Neumann algebras is thought of as being the proper analogue of nuclearity of $C^*$-algebras. We will later see that our following definition is a much stronger condition than injectivity and does not even include the class of abelian von Neumann algebras.

\begin{defn}
A von Neumann algebra $M$ has the {\it weak* tensor uniqueness property} (or {\it WTU property}) if for any von Neumann algebra $N$, any two weak* tensor products $M\overline\otimes_\alpha N$ and $M\overline\otimes_\beta N$ of $M$ and $N$ are equivalent.
\end{defn}

\begin{prop}\label{type I}
Let $H$ be a Hilbert space. Then $B(H)$ has the WTU property.
\end{prop}

\begin{proof}
Let $N$ be a von Neumann algebra and suppose that $B(H)\overline\otimes_\alpha N\subset B(K)$ is a weak* tensor product of $B(H)$ and $N$. Then $\alpha|_{B(H)}$ is unitarily equivalent to an amplification map of $B(H)$. So we may assume that $K=H\otimes\ell^2(I)$ for some index set $I$ and
$$\alpha(a\otimes 1)=a\otimes 1\in B(H)\overline{\otimes}B(\ell^2(I))=B(K)$$
for every $a\in B(H)$. Since $\alpha(1\otimes N)$ commutes with $\alpha(B(H)\otimes 1)=B(H)\otimes 1$, we have that $\alpha(1\otimes N)\subset 1\otimes B(\ell^2(I))$. It follows that $B(H)\overline\otimes_\alpha N$ is equivalent to the normal spatial tensor product $B(H)\overline\otimes N$.
\end{proof}

\section{Remarks on the maximal weak* tensor product}

Before continuing onto the main results of this paper, we pause to record a couple connections between the maximal weak* tensor product and related constructions. We begin by establishing a connection with the maximal tensor product of $C^*$-algebras.

\begin{proposition}
Let $A$ and $B$ be $C^*$-algebras. Then the identity map on $A\odot B$ extends to a normal $*$-isomorphism $(A\otimes_{\max} B)^{**}\cong A^{**}\wm B^{**}$.
\end{proposition}

\begin{proof}
We first assume that $A$ and $B$ are unital.

Let $\alpha$ be a $*$-representation of $A^{**}\odot B^{**}$  such that $A^{**}\oo_\alpha B^{**}= A^{**}\wm B$. Then, since $a\otimes b\mapsto \alpha(a\otimes b)$ extends to a $*$-representation from $A\otimes_{\max} B\to A^{**}\wm B^{**}$, we have that the identity map on $A\odot B$ extends to a normal $*$-homormophism from $(A\otimes_{\max} B)^{**}\to A^{**}\wm B^{**}$.

Now let $\pi_u\fn A\odot B\to B(H_u)$ be the universal representation of $A\otimes_{\max} B$. Then $\pi_u|_A$ extends to a normal $*$-homomorphism $A^{**}\to B(H_u)$ and $\pi_u|_B$ extends to a normal $*$-homomorphism $B^{**}\to B(H_u)$. As the ranges of $\pi_u|_{A}$ and $\pi_u|_{B}$ commute, the identity map on $A\odot B$ extends to a normal $*$-homomorphism $A^{**}\wm B^{**}\to (A\otimes_{\max} B)^{**}$. Hence, the identity map on $A\odot B$ extends to a normal $*$-isomorphism $(A\otimes_{\max} B)^{**}\cong A^{**}\wm B^{**}$.

In the case that $A$ and $B$ are not necessarily unital, then the canonical embedding of $A^{**}\odot B^{**}$ into $A_I^{**}\odot B_I^{**}$ (where $A_I$ and $B_I$ are the unitizations of $A$ and $B$) is readily verified to extend to a binormal embedding of $A^{**}\wm B^{**}$ into $A_I^{**}\wm B_I^{**}$ since every $*$-representation of $A\odot B$ canonically extends to a $*$-representation of $A_I\odot B_I$. Therefore the idenitity map on $A\odot B$ extends to a normal $*$-isomorphism $(A\otimes_{\max} B)^{**}\cong A^{**}\wm B^{**}$ for any $C^*$-algebras $A$ and $B$.
\end{proof}

Let $M$ and $N$ be von Neumann algebras. The binormal $C^*$-norm $\|\cdot\|_{\mathrm{bin}}$ of $M\odot N$ is defined by
$$\left\|x\right\|_{\mathrm{bin}}=\sup\{\varphi(x^*x)^{1/2} : \varphi\in S(M\odot N) \tn{ and } (a,b)\mapsto \varphi(a\otimes b)\tn{ is separately weak* continuous}\}$$
where $S(M\odot N)$ is the set of states on $M\odot N$, i.e. the set of linear maps $\varphi\fn M\odot N\to \C$ such that $\varphi(1)=1$ and $\varphi(x^*x)\geq 0$ for every $x\in M\odot N$.
This $C^*$-tensor norm was defined and studied by Effros and Lance in \cite{el}. The main result of their paper on this norm is that a von Neumann algebra $M$ has the property that the binormal $C^*$-norm $\|\cdot\|_{\mathrm{bin}}$ agrees with the minimal $C^*$-norm $\|\cdot\|_{\min}$ on $M\odot N$ for all choices of von Neumann algebras $N$ if and only if $M$ is semidiscrete (see \cite[Theorem 4.1]{el}). We will next observe observe that the norm on $M\odot N$ arising from the inclusion in $M\wm N$ is exactly the binormal $C^*$-norm.

\begin{prop}
Let $M$ and $N$ be von Neumann algebras and $\|\cdot\|_{w^*-\tn{max}}$ denote the norm on $M\odot N$ arising from the inclusion into $M\wm N$. Then $\|\cdot\|_{w^*-\tn{max}}=\|\cdot\|_{\tn{bin}}$.
\end{prop}

\begin{proof}
Let $\varphi\in S(M\odot N)$ and suppose that $\varphi|_{M}$ and $\varphi|_{N}$ are weak* continuous. Denote the GNS representation of $\varphi$ by $\pi_\varphi\fn M\odot N\to B(H_\varphi)$. We claim that the restrictions $\pi_\varphi|_{M}$ and $\pi_\varphi|_{N}$ define normal maps on $M$ and $N$, respectively. Indeed, let $x=\sum_{j=1}^n a_i\otimes b_i$ and $y=\sum_{k=1}^m a_i'\otimes b_i'$ be elements of $M\odot N$. Denoting the images of $x$ and $y$ in $H_\varphi$ by $\Lambda(x)$ and $\Lambda(y)$, respectively, we then have that the map
$$M\times N\ni (a,b)\mapsto \lla \pi_\varphi(a\otimes b) \Lambda(x),\Lambda(y)\rra=\sum_{j,k} \varphi( ((a_k')^*aa_j)\otimes ((b_k')^*bb_j))$$
is separately weak* continuous. Since $\Lambda(M\odot N)$ is dense in $H_\varphi$, it follows that $\pi_\varphi|_M$ is WOT-WOT continuous on the unit ball of $M$. Therefore $\pi_\varphi|_{M}$ is normal and, similarly, $\pi_\varphi|_{N}$ is also normal. So $\|\cdot\|_{\tn{bin}}\leq \|\cdot\|_{w^*-\tn{max}}$.

Now let $\alpha\fn M\odot N\to B(H)$ be a $*$-representation which satisfies the conditions required for a weak* tensor product. It is clear that $(a,b)\mapsto \lla \alpha(a\otimes b)\xi,\eta\rra$ is separately weak* continuous for all $\xi,\eta\in H$ and, so, $\|\cdot\|_{w^*-\tn{max}}\leq \|\cdot\|_{\tn{bin}}$. 
\end{proof}

\section{Characterization of the weak* tensor uniqueness property}\label{WTU section}

In this section we study the WTU property and give a complete characterization of the von Neumann algebras with this property. We have already seen in Section \ref{definitions section} that a factor $M$ has the WTU property if and only if $M$ is of type I. We will show in this section that a von Neumann algebra $M$ has the WTU property if and only if $M$ is a direct product of type I factors.

\begin{lemma}\label{products}
Let $\{M_i : i\in I\}$ be a set of von Neumann algebras. Then $M:=\prod_{i\in I}M_i$ has the WTU property if and only if $M_i$ has the WTU property for each $i\in I$.
\end{lemma}

\begin{proof}
We first suppose that there exists an index $j\in I$ so that $M_{j}$ fails to have the WTU property. Then there exist two inequivalent weak* tensor products $M_{j}\overline\otimes_\alpha N$ and $M_{j}\overline\otimes_\beta N$ for some von Neumann algebra $N$. Let $M_i\overline\otimes_{\alpha_i} N$ and $M_i\overline\otimes_{\beta_i} N$ for $i\in I$ be arbitrary weak* tensor products such that $\alpha_j=\alpha$ and $\beta_j=\beta$. Then $M\overline\otimes_{\bigoplus_i\alpha_i} N$ and $M\overline\otimes_{\bigoplus\beta_i} N$ are inequivalent weak* tensor products since $\alpha$ is not quasiequivalent to $\beta$ implies that $\bigoplus_i \alpha_i$ is not quasiequivalent to $\bigoplus_i \beta_i$.

Next, for each index $j$ let $e_j\fn \prod_{i\in I} M_i\to M_j$ be the restriction to the $j$th component and suppose that $M_i$ has the WTU property for every $i\in I$. Let $N$ be an arbitrary von Neumann algebra and $\prod_{i\in I}(M_i)\overline\otimes_\alpha N$ a weak* tensor product of $\prod_i M_i$ and $N$. Then $e_j\otimes 1\in (M\overline\otimes_\alpha N)'$ since $e_j\otimes 1$ commutes with $a\otimes b$ for all $a\in M$, $b\in N$ and $M\odot N$ is weak* dense in $M\overline\otimes_\alpha N$. Since $\sum_{i\in I} e_i\otimes 1=1$, it follows that
$$M\overline\otimes_\alpha N=\prod_{i\in I} (e_i\otimes 1)M\overline\otimes_\alpha N=\prod_{i\in I} M_i\overline\otimes_{\alpha_i} N$$
where $\alpha_i$ denotes $\alpha|_{M_i\odot N}$. But $M_i\overline\otimes_{\alpha_i} N=M\overline{\otimes}N$ for each $i\in I$ by the WTU property. Therefore $M\overline\otimes_\alpha N=\prod_{i\in I} M_i\overline{\otimes}N$ and, hence, we conclude that $M$ has the WTU property. 
\end{proof}

Choi and Effros showed in \cite[Lemma 2.1]{el} that a von Neumann algebra $M\subset B(H)$ is injective if and only if the multiplication map $m\fn M\odot M'\to B(H)$ is continuous with respect to the minimal tensor product norm. It is interesting to see in the following theorem that an analogous characterization of the WTU property also holds.

\begin{theorem}
The following are equivalent for a von Neumann algebra $M\subset B(H)$.
\begin{itemize}
\item[(i)] $M$ has the WTU property;
\item[(ii)] The weak* tensor products $M\overline{\otimes}M'$ and $M\wm M'$ are equivalent;
\item[(iii)] The multiplication map $m\fn a\otimes b\mapsto ab$ extends to a normal $*$-homomorphism from $M\oo M'$ to $B(H)$;
\item[(iv)] $M$ is of the form $\prod_{i\in I} B(H_i)$ for some choices of Hilbert spaces $H_i$.
\end{itemize}
\end{theorem}

\begin{proof}
(i) $\Rightarrow$ (ii): Trivial.

(ii) $\Rightarrow$ (iii): Suppose that the multiplication map $m\fn M\odot M'\to B(H)$ does not extend to a normal $*$-homomorphism $\widetilde m\fn M\overline{\otimes}M'\to B(H)$. Then $m$ is not quasi-contained in the spatial embedding $\iota\fn M\odot M'\to B(H\otimes H)$. Defining $\alpha=\iota\oplus m$, we have that $M\overline\otimes_{\alpha} M'$ is a nonequivalent weak* tensor product to $M\overline\otimes M'$ with the property that the identity map on $M\odot M'$ extends to a normal $*$-homomorphism from $M\overline\otimes_{\alpha}M'$ to $M\overline{\otimes}M'$. It follows that $M\overline{\otimes}M'$ and $M\wm M'$ are inequivalent weak* tensor products.

(iii) $\Rightarrow$ (iv): Suppose that $m$ extends to a normal $*$-homomorphism $\widetilde m\fn M\overline\otimes M'\to B(H)$. We claim that $Z(M)=M\cap M'$ must then be of the form $\ell^\infty(S)$ for some set $S$. Indeed, suppose otherwise. Then, without loss of generality, we may assume that $Z(M)=L^\infty(X,\mu)$ for some locally compact space $X$ equipped with a necessarily nondiscrete positive Radon measure $\mu$. Since $\mu$ is nondiscrete, there exists a Borel subset $E$ of $X$ such that $\mu(E)>\sum_{x\in E}\mu(\{x\})$. By inner regularity of $\mu$, we can then find a compact subset $K$ of $X$ so that $\mu(K)>\sum_{x\in K} \mu(\{x\})$. We will show that $\mu$ being nondiscrete leads to a contradiction of the normality of $\widetilde m$.

Let $f_1,\ldots,f_n$ and $g_1,\ldots,g_n$ be functions in $C(K)\subset L^\infty(X)$. Then
$$m(f_1\otimes g_1+\ldots+f_n\otimes g_n)(x)=(f_1\otimes g_1+\ldots+f_n\otimes g_n)(x,x)$$
for $x\in X$. By norm continuity, it follows that $\widetilde m(f)(x)=f(x,x)$ for almost every $f\in C(K\times K)=C(K)\otimes_{\min} C(K)\subset L^\infty(X)\oo L^\infty(X)$.

Choose a sequence of descending relatively open subsets $U_1\supset U_2\supset\ldots$ of $K\times K$ so that $U_n$ contains $\Delta:=\{(x,x) : x\in K\}$ for every $n$ and $(\mu\times\mu)(U_n)\to (\mu\times\mu)(\Delta)$ as $n\to\infty$. By Urysohn's lemma, there exists functions $f_1,f_2,\ldots\in C(K\times K)\subset L^\infty(X)\oo L^\infty(X)$ mapping into $[0,1]$ such that $f_n(x,x)=1$ for every $x\in K$ and $f_n(x,y)=0$ for $(x,y)\not\in U_n$. Then, by Lebesgue's dominated convergence theorem,
$$\int f_n(x,y)g(x,y)\,d\mu(x)d\mu(y)\to \int_{\Delta} g(x,y)\,d\mu(x)d\mu(y)$$
as $n\to \infty$ for every $g\in L^1(X\times X)$. Hence, $f_n\to 1_\Delta$ in the weak* topology. Then, since $\widetilde m$ is normal and $\widetilde m(f_n)=1_K$ for every $n$, we have that $\widetilde m(1_\Delta)=1_K$.

Next, define a net of elements in $L^\infty(X\times X)$ indexed under reverse inclusion by the finite subsets $F$ of $K$ by $h_F=\sum_{x\in F}1_{\{(x,x)\}}=\sum_{x\in F}1_{\{x\}}\otimes 1_{\{x\}}$. Then
\begin{eqnarray*}
\int h_F(x,y)g(x,y)\,d\mu(x)d\mu(y) &=& \sum_{x\in F} g(x,x)\mu(x)^2 \\
&\to & \sum_{x\in K} g(x,x)\mu(x)^2 \\
&=& \int_{\Delta} g(x,y)\,d\mu(x)d\mu(y)
\end{eqnarray*}
in the limit as $F\to K$ for every $g\in L^1(X\times X)$. Hence, $h_F$ converges weak* to $1_\Delta$.

Observe that
$$ \int_K m(h_F)(x)\,d\mu(x)=\sum_{x\in F} \mu(\{x\})\to \sum_{x\in K}\mu(\{x\})$$
as $F\to K$. Recall that $\mu(K)>\sum_{x\in K} \mu(\{x\})$. It follows that $\widetilde m(h_F)$ does not converge to $1_K$. This contradicts the normality of $\widetilde m$ and, so, we conclude that $\mu$ must be a discrete measure. Hence, $Z(M)$ is of the form $\ell^\infty(S)$ for some set $S$.

Since $Z(M)$ is of the form $\ell^\infty(S)$, there exists a set $\{e_i : i\in S\}$ of minimal central projections in $M$ such that $\sum_{i\in S} e_i=1$. Then $M=\prod_{i\in S} e_i M$ where each term $e_iM\subset B(e_iH)$ is a factor by the minimality of $e_i$. Therefore the desired result follows from Lemma \ref{products} since a factor has the WTU property if and only if it is of type I.

(iv) $\Rightarrow$ (i): This is clear from Proposition \ref{type I} and Lemma \ref{products}.
\end{proof}

\section{Weak* tensor product completions of $L^\infty(\R)\odot L^\infty(\R)$}

This section is dedicated to studying weak* tensor product completions of $L^\infty(\R)\odot L^\infty(\R)$. The main result theorem of this section is the construction of $2^{\mf c}$ distinct such weak* tensor product completions. As a consequence of our constructions, we also find that the weak* maximal tensor product $L^\infty(\R)\wm L^\infty(\R)$ does not have a separable predual despite $L^\infty(\R)$ having separable predual $L^1(\R)$.

The approach that we take in constructing weak* tensor products is to use the fact that $L^\infty(\R)$ contains a weak* dense copy of the continuous bounded functions $C_b(\R)$ (we choose to use $C_b(\R)$ over $C_0(\R)$ for convenience since $C_b(\R)$ is unital) and, hence, any weak* tensor product completion $L^\infty(\R)\oo_\alpha L^\infty(\R)$ of $L^\infty(\R)\odot L^\infty(\R)$ must contain a weak* dense copy of $C_b(\R)\otimes_{\min} C_b(\R)$. In particular, this implies that if $\pi\fn C_b(\R)\otimes_{\min}C_b(\R)\to B(H)$ is a $*$-representation such that the restrictions $\pi|_{C_b(\R)\otimes 1}$ and $\pi|_{1\otimes C_b(\R)}$ are quasi-contained in the canonical representation $\sigma\fn C_b(\R)\to L^\infty(\R)$, then $\pi$ extends uniquely to a normal $*$-representation $\widetilde\pi\fn L^\infty(\R)\wm L^\infty(\R)\to B(H)$.

We will identify $C_b(\R)\otimes_{\min} C_b(\R)$ as being a $C^*$-subalgebra of $C_b(\R\times\R)$. For every $x\in\R$, define an (unbounded) measure $\mu_x$ on $\R\times\R$ by
$$ \int f\,d\mu_x=\int_\R f(y,x+y)\,dy $$
for $f\in C_c(\R\times\R)$.
Further, we let 
$$\pi_x\fn C_b(\R)\otimes_{\min} C_b(\R)\to L^\infty(\R\times\R,\mu_x)$$
 be the natural inclusion.
These $*$-representations $\pi_x$ will be our building blocks in constructing many weak* tensor product completions of $L^\infty(\R)\odot L^\infty(\R)$.

\begin{lemma}\label{subrep}
For every $x\in\R$, the $*$-representations $\pi_{x}|_{C_b(\R)\otimes 1}$ and $\pi_{x}|_{1\otimes C_b(\R)}$ are unitarily equivalent to the canonical representation $\sigma\fn C_b(\R)\to L^\infty(\R)$.
\end{lemma}

\begin{proof}
Define a map from $U\fn L^2(\mu_x)\to L^2(\R)$ by $U(g)(y)=g(y,x+y)$. Then $U$ is clearly a well defined surjective isometry. Further, we observe that if $f\in C_b(\R)$ and $\xi\in L^2(\mu_x)$, then
$$ U\big(\pi_x(f\otimes 1)\xi\big)(y)=f(y)\xi(y,x+y)=\big(\sigma(f)U(\xi)\big)(y)$$
for almost every $y\in\R$. So we have shown that $\pi_{x}|_{C_b(\R)\otimes 1}$ is unitarily equivalent to a subrepresentation of $\sigma$. A similar argument shows that the same holds for $\pi_{x}|_{1\otimes C_b(\R)}$.
\end{proof}


The author wishes to thank Nico Spronk for suggesting the following lemma.

Let $X$ be a measure space. We say that a measure $\mu$ on $X$ is absolutely continuous with respect to a family of measures $\{\nu_i : i\in I\}$ on $X$ if $\nu_i(E)=0$ for every $i\in I$ implies that $\mu(E)=0$ whenever $E\subset X$ is a measurable subset.

\begin{lemma}\label{measure spaces}
Let $X$ be a locally compact space, and $\mu$ and $\{\nu_i : i\in I\}$ be Radon measures on $X$. Denote the canonical inclusions of $C_0(X)$ in $L^\infty(\mu)$ and $L^\infty(\nu_i)$ by $\pi$ and $\sigma_i$, respectively. Then $\pi$ is quasi-contained in $\bigoplus_{i\in I}\sigma_i$ if and only if $\mu$ is absolutely continuous with respect to $\{\nu_i : i\in I\}$.
\end{lemma}

\begin{proof}
If $\mu$ is absolutely continuous with respect to $\{\nu_i : i\in I\}$, then the identity map on the bounded measurable functions $\mc M_b(X)$ of $X$ clearly extends to a surjective $*$-isomorphism from $\big(\bigoplus_{i\in I}\sigma_i\big)(\mc M_b(X))$ to $\pi(\mc M_b(X))$. Since $\pi(\mc M_b(X))=L^\infty(\mu)=\pi(C_b(X))''$ and $\big(\bigoplus_{i\in I} \sigma_i\big)(\mc M_b(X))=\big(\bigoplus_{i\in I} \sigma_i\big)(C_b(X))''$, this implies that $\pi$ is quasi-contained in $\bigoplus_{i\in I}\sigma_i$.

Next we suppose towards a contradiction that $\pi$ is quasi-contained in $\bigoplus_{i\in I}\sigma_i$, but $\mu$ is not absolutely continuous with respect to $\{\nu_i : i\in I\}$. Since $\pi$ is quasi-contained in $\bigoplus_{i\in I}\sigma_i$, there exists a cardinal $\omega$ such that $\pi$ is unitarily equivalent to a subrepresentation of the amplification $\omega\cdot \bigoplus_{i\in I} \sigma_i$. Note that since $\mu$ is not absolutely continuous with respect to $\{\nu_i : i\in I\}$, there exists a compact set $K$ such that $\mu(K)>0$ but $\nu_i(K)=0$ for every $i\in I$ by inner regularity. So the function $\varphi\fn C_b(X)\to \C$ defined by
$$\varphi(f)=\int_K f\,d\mu=\lla \pi(f) \xi,\xi\rra,$$
where $\xi\in L^2(\mu)$ is the characteristic function $1_K$, is a vector state of $\varphi$. This implies that $\varphi$ must also be a vector state of $\omega\cdot \bigoplus_{i\in I} \sigma_i$.

Observe that the Hilbert space on which $\omega\cdot \bigoplus_{i\in I} \sigma_i$ acts is
$$ \bigoplus_{i\in I} L^2(\nu_i)^{\oplus \omega}.$$
So, assuming that $\varphi$ is a vector state of $\omega\cdot \bigoplus_{i\in I} \sigma_i$, then we can approximate $\varphi$ arbitrarily well in norm by maps $\psi\fn C_b(X)\to \C$ of the form
$$ \psi(f)=\sum_{j=1}^m\sum_{k=1}^{n_j}\lla \sigma_{i_j}(f)\xi_{jk},\xi_{jk}\rra =\sum_{j=1}^m\sum_{k=1}^{n_j} \int f|\xi_{jk}|^2\,d\nu_{i_j}$$
for choices of $m,n_j\in \N$, $i_j\in I$, and $\xi_{jk}\in L^2(\mu_{i_j})$. We will show that this is not possible.

By outer regularity, we can choose a decreasing sequence of open subsets $U_1\supset U_2\supset \cdots$ of $X$ containing $K$ such that $\nu_{i_j}(U_p)\to 0$ as $p\to \infty$ for all $j=1,\ldots, m$. By Urysohn's lemma, there exists functions $f_p\fn X\to [0,1]$ such that $f_p(x)=1$ for $x\in K$ and $f_p(x)=0$ for $x\not\in U_p$. So, by Lebesgue's dominated convergence theorem,
$$ \int f_p|\xi_{jk}|^2\,d\mu_{i_j}\to 0$$
as $p\to \infty$ and all $j=1,\ldots, m$ and $k=1,\ldots, n_j$. Hence, $\psi(f_p)\to 0$ as $p\to \infty$. Since $\varphi(f_p)=\mu(K)$ for every $p$, this implies that $\|\psi -\varphi\|\geq \mu(K)>0$. This contradicts that we should be able to approximate $\varphi$ arbitrarily well with functionals with the form of $\psi$. So we conclude that $\pi$ is not quasi-contained in $\bigoplus_{i\in I}\sigma_i$.
\end{proof}

\begin{corollary}\label{not quasiequiv}
For each $x\in\R$, the representation $\pi_x$ is not quasi-contained in
$$\bigoplus_{x'\neq x}\pi_{x'}\oplus \sigma,$$
where $\sigma\fn C_b(\R)\otimes_{\min} C_b(\R)\to L^\infty(\R)\oo L^\infty(\R)$ denotes the canonical inclusion. 
\end{corollary}

Before proving the main theorem of this section, we pause to note a further corollary of the proof of Lemma \ref{measure spaces}.

\begin{corollary}
The von Neumann algebra $L^\infty(\R)\wm L^\infty(\R)$ does not have a separable predual.
\end{corollary}

\begin{proof}
For each $x\in\R$, let $\varphi_x\fn C_b(\R)\otimes_{\min} C_b(\R)\to\C$ be the function defined by
$$\varphi_x(f)=\int_{[0,1]} f(y,x+y)\,dy=\lla \pi_x(f) \xi,\xi\rra$$
where $\xi\in L^2(\mu_x)$ is the characteristic function $1_{[0,1]\times[x,1+x]}$. Then, as  $\varphi_x$ is a vector state of $\pi_x$, it follows that $\varphi_x$ is in the predual of $L^\infty(\R)\wm L^\infty(\R)$. So the predual of $L^\infty(\R)\wm L^\infty(\R)$ cannot be separable since, by the proof of Lemma \ref{measure spaces}, $\|\varphi_x-\varphi_y\|\geq 1$ for all distinct $x,y\in \R$.
\end{proof}

\begin{theorem}\label{2^c}
$L^\infty(\R)\odot L^\infty(\R)$ admits $2^{\mf c}$ nonequivalent weak* tensor products.
\end{theorem}

\begin{proof}
Let $\sigma\fn C_b(\R)\otimes_{\min} C_b(\R)\to L^\infty(\R)\oo L^\infty(\R)$ be the canonical inclusion. For each subset $S\subset \R$, define $\alpha_S=\bigoplus_{x\in S}\pi_x\oplus \sigma$. Then, by Lemma \ref{subrep}, the restrictions $\alpha_S|_{C_b(\R)\otimes 1}$ and $\alpha_S|_{1\otimes C_b(\R)}$ are each quasi-equivalent to the canonical inclusion $C_b(\R)\to L^\infty(\R)$. Further, since $\alpha_S$ contains $\sigma$ as a subrepresentation, we conclude that each $\alpha_S$ extends to a $*$-representation $\widetilde\alpha_S$ of $L^\infty(\R)\odot L^\infty(\R)$ which satisfies the conditions required to construct a weak* tensor product. It is clear from Corollary \ref{not quasiequiv} that these weak* tensor products $L^\infty(\R)\overline\otimes_{\widetilde\alpha_S}L^\infty(\R)$ are pairwise nonequivalent for $S\subset \R$. Hence, $L^\infty(\R)\odot L^\infty(\R)$ admits $2^{\mf c}$ nonequivalent weak* tensor products.
\end{proof}

\begin{remark}
Let $M$ be any infinite dimensional abelian von Neumann algebra with separable predual. Recall that then $M=\ell^\infty(\N)$, $M=L^\infty(\R)$ or $M=L^\infty(\R)\oplus \ell^\infty(S)$ where $S=\{1,\ldots,n\}$ for some $n\in \N$ or $S=\N$. In particular, it follows that if $M$ does not have the WTU property (or, equivalently, $M\neq \ell^\infty(\N)$), then $M\odot M$ admits $2^{\mf c}$ distinct weak* tensor product completions.
\end{remark}





\section*{Acknowledgements}
The author wishes to thank his advisor Nico Spronk for carefully reading this paper and his helpful comments. The author was partially supported by an NSERC Postgraduate Scholarship.

\end{document}